\documentclass[11pt,oneside]{article}
\usepackage{amssymb,amsmath,url}
\usepackage[english]{babel}
\usepackage{pstricks,pst-node}
\renewcommand{\le}{\leqslant}
\renewcommand{\ge}{\geqslant}

\newcommand{\rad}{\mathrm{rad}}

\newcommand{\RR}{\mathbb{R}}
\newcommand{\ZZ}{\mathbb{Z}}
\newcommand{\QQ}{\mathbb{Q}}
\newcommand{\PP}{\mathbb{P}}

\newcommand{\NN}{\mathbb{N}}
\newcommand{\FF}{\mathbb{F}}

\newcommand{\LLL}{\mathcal{L}}
\newcommand{\NNN}{\mathcal{N}}

\newcommand{\SSS}{\mathcal{S}}

\newcommand{\vu}{\mathbf{u}}

\newcommand{\va}{\mathbf{a}}

\newtheorem{lemma}{Lemma}
\newtheorem{theorem}{Theorem}
\newtheorem{proposition}{Proposition}

\newtheorem{problem}{Problem}

\newenvironment{proof}{\textsc{Proof. }}{\ \newline\hspace*{\fill}$\boxtimes$}

\title{On spectrum of irrationality exponents of Mahler numbers}

\author{Dzmitry Badziahin\footnote{Univerity of Sydney, School of Mathematics and Statistics, NSW 2006}}

\begin{document}

\maketitle

\begin{abstract}
We consider Mahler functions $f(z)$ which solve the functional equation
$f(z) = \frac{A(z)}{B(z)} f(z^d)$ where $\frac{A(z)}{B(z)}\in \QQ(z)$ and
$d\ge 2$ is integer. We prove that for any integer $b$ with $|b|\ge 2$
either $f(b)$ is rational or its irrationality exponent is rational. We
also compute the exact value of the irrationality exponent for $f(b)$ as
soon as the continued fraction for the corresponding Mahler function is
known. This improves the result of Bugeaud, Han, Wei and
Yao~\cite{bhwy_2016} where only an upper bound for the irrationality
exponent was provided.

\end{abstract}

\section{Introduction}

Consider a Laurent series $f(z)\in\QQ((z^{-1}))$. It is called a {\it Mahler
function} if for any $z$ inside the disc of convergence of $f$ it satisfies
the equation of the form
$$
\sum_{i=0}^n P_i(z)f(z^{d^i})=Q(z)
$$
for some integers $n\ge 1, d\ge 2$, and polynomials $P_0,\ldots, P_n, Q\in
\FF[z]$ with $P_0P_n\neq 0$. The values $f(b)$ for integers $b$ inside the
disc of convergence of $f$ are called {\em Mahler numbers}. In this paper we
investigate the following question:

\begin{problem}\label{probc}
Find the set $\LLL_M$ of irrationality exponents of irrational Mahler
numbers.
\end{problem}

Sometimes we will call the set $\LLL_M$ \emph{the spectrum} of irrationality
exponents of Mahler numbers. Recall that the irrationality exponent of a real
number $\xi$ is the supremum of real numbers $\mu$ such that the inequality
$$
\left|\xi - \frac{p}{q}\right|<q^{-\mu}
$$
has infinitely many rational solutions $p/q$. This is one of the most
important approximational properties of real numbers indicating how well
$\xi$ is approached by rationals. Note that by the classical Dirichlet
approximation theorem we always have $\mu(\xi)\ge 2$.

Similar questions have been recently risen by several authors. In 2008,
Bugeaud~\cite{bugeaud_2008} proved that for any rational $\omega$ there are
infinitely many automatic numbers with the irrationality exponent equal to
$\omega$. It is well known \cite{becker_1994} that automatic numbers are also
Mahler, therefore Bugeaud's result straightforwardly implies that $\LLL_M$
contains all rational numbers not smaller than two. Later in 2009,
Adamczewski and Rivoal~\cite{adamczewski_rivoal_2009} commented on that
result with the following problem:

\begin{problem}\label{proba}
Is it true or not that the irrationality exponent of an automatic number is
always a rational number?
\end{problem}

Bugeaud, Krieger and Shallit~\cite{bgs_2011} extended the question to the set
of morphic numbers. In particular, they show that the spectrum of
irrationality exponents of morphic numbers, on top of $\QQ$, contains every
Perron number $\omega\ge 2$. Recall that a Perron number is a positive real
algebraic integer, which is greater in absolute value than all of its
conjugates. With respect to this result the following problem was posed:

\begin{problem}\label{probb}
Determine the set of irrationality exponents of morphic numbers. In
particular, is it true that the irrationality exponent of a morphic number is
always algebraic?
\end{problem}

In some sources~\cite{gwz_2016}, Problems~\ref{proba} and~\ref{probb} were
referred as conjectures.

In this paper we restrict our research to Mahler functions $f(z)$ which solve
the following functional equation
\begin{equation}\label{def_feq}
f(z) = \frac{A(z)}{B(z)} f(z^d),\qquad A,B\in\QQ[z],\; B\neq 0.
\end{equation}

\begin{theorem}\label{th2}
Let $f(z)\in \QQ((z^{-1}))$ be a solution of~\eqref{def_feq} and $b\in\ZZ$ be
inside the disk of convergence of $f(z)$. Assume that
$A(b^{d^m})B(b^{d^m})\neq 0$ for all $m\in\ZZ_{\ge 0}$ and $f(b)$ is
irrational. Then the irrationality exponent of $f(b)$ is a rational number.
\end{theorem}

In other words, Theorem~\ref{th2} shows that Mahler numbers $f(b)$ for
solutions of~\eqref{def_feq} do not give any extra contribution to the
spectrum $\LLL_M$ on top of the numbers constructed by Bugeaud.

The key ingredient in the proof of Theorem~\ref{th2} is that $\mu(f(b))$ can
be computed from the information on the continued fraction of the Laurent
series $f(z)$. This phenomenon was observed by Guo, Wu and
Wen~\cite{gww_2014} and then developed by Bugeaud, Han, Wen and
Yao~\cite{bhwy_2016}. They computed the upper bound of the irrationality
exponent of $f(b)$ depending on the distribution of non-zero Hankel
determinants. In this paper we make their result much stronger by computing
the precise value of $\mu(f(b))$.

Recall that a Laurent series $f(z)\in \QQ((z^{-1}))$ admits the continued
fraction expansion
$$
f(z) = [a_0(z), a_1(z), \ldots, a_k(z),\ldots],
$$
where $a_i\in \QQ[z]$. It is finite if and only if $f(z)$ is a rational
function. As in the case of real numbers, we call the rational function
$[a_0(z), a_1(z), \ldots, a_n(z)] = \frac{p_k(z)}{q_k(z)}$ {\it $n$th
convergent of $f$}. Assuming that $p_k$ and $q_k$ are coprime, denote by
$d_k$ the degree of the denominator $q_k$.

\begin{theorem}\label{th1}
Let $f(z)\in \QQ((z^{-1}))$ be a Laurent series which solves~\eqref{def_feq}
and is not a rational function. Let $b\in \ZZ$ with $|b|\ge 2$ be inside the
disk of convergence of $f$ such that $A(b^{d^m})B(b^{d^m}) \neq 0$ for any
$m\in\ZZ_{\ge 0}$. Then the irrationality exponent of $f(b)$ equals
\begin{equation}\label{th1_eq}
\mu(f(b)) = 1 + \limsup_{k\to\infty}\frac{d_{k+1}}{d_k}.
\end{equation}
\end{theorem}

Unfortunately, Theorem~\ref{th1} does not always allow to compute $\mu(f(b))$
based only on the knowledge of polynomials $A$ and $B$ because the
formula~\eqref{th1_eq} requires knowledge of the whole continued fraction of
$f$ which is usually a difficult task. However in many cases, as soon as we
know that $\mu(f(b))>2$ we can compute the precise value of the irrationality
exponent of $f(b)$ after computing only finitely many convergents of $f$. We
demonstrate the method by computing the irrationality exponents of all Mahler
numbers from~\cite{badziahin_2017} for which we know that $\mu(f(b))>2$.

\begin{theorem}\label{th3}
Let $q_{\va} = g_{a_1,a_2}(z)\in \ZZ((z^{-1}))$ solve the equation
\begin{equation}\label{th3_eq}
g_\va (z) = (z^2 +a_1z + a_2) g_\va(z^3);\quad a_1,a_2\in\ZZ.
\end{equation}
Then for any integer $b$, $|b|\ge 2$ we have
\begin{itemize}
\item[(a)] for any $s\in \ZZ$ if $f_{s,s^2}(b)$ is irrational then
    $\mu(f_{s,s^2}(b))=3$;
\item[(b)] for any $s\in \ZZ$ if $f_{s^3,-s^2(s^2+1)}(b)$ is irrational
    then $\mu(f_{s^3,-s^2(s^2+1)}(b))=3$;
\item[(a)] if $f_{\pm2,1}(b)$ is irrational then
    $\mu(f_{\pm2,1}(b))=\frac{12}{5}$.
\end{itemize}

\end{theorem}

\section{Useful estimate of irrationality exponent}

The following proposition is a modification of Lemma~4.1
from~\cite{adamczewski_rivoal_2009} which we will need in our proofs. But it
may be of independent interest.

\begin{proposition}\label{prop1}
Let $\alpha\in\RR$. Assume that there exist two sequences
$(\frac{p_n}{q_n})_{n\in \NN}\in\QQ$ and $(\frac{p'_n}{q'_n})_{n\in
\NN}\in\QQ$ of rational approximations to $\alpha$ and three sequences
$\theta_n$, $\delta_n$ and $\tau_n$ of real numbers with $\theta_n\ge 1,
\delta_n> 0$, $\tau_n>0$ such that
\begin{itemize}
\item{(a)} $q'_n\ll q_n^{\theta_n}$;
\item{(b)} $\displaystyle \left|\alpha-\frac{p_n}{q_n}\right| \asymp
    q_n^{-1-\delta_n}$;\quad $\displaystyle
    \left|\alpha-\frac{p'_n}{q'_n}\right| \asymp (q'_n)^{-1-\tau_n}$;
\item{(c)} $(q'_n)^{\tau_n}\gg q_{n+1}^{\delta_{n+1}}$, \quad and\quad
    $q_n^{\delta_n}\to\infty$ as $n\to\infty$.
%\item{(d)} $(q'_n)^{\tau_n}\ge q_{n+1}^{\delta_{n+1}}$.
\end{itemize}
for all $n\in\NN$. Then we have that
\begin{equation}\label{prop1_eq1}
\mu(\alpha) \le \limsup_{n\to\infty} \max\left
\{1+\frac{\theta_n}{\delta_n}, \frac{(1+\tau_n)\theta_n}{\delta_n}\right\}.
\end{equation}
\end{proposition}

The immediate corollary of this proposition is that if the sequences
$\theta_n$ and $\delta_n$ satisfy $\theta_n/\delta_n \to 1$ as $n\to\infty$
and $\tau_n\ge 1$ then the sequence of approximations $p'_n/q'_n$ to $\alpha$
is nearly optimal, i.e. $$\mu(\alpha) = \limsup_{n\to \infty} (1+\tau_n).$$
%Moreover, in this case we can compute the precise value of the irrationality
%exponent for $\alpha$.

\begin{proof}
Denote by $c$ the implied constant from Condition~(b), i.e. the constant
which satisfies the inequality
\begin{equation}\label{prop1_eq2}
\left|\alpha-\frac{p'_n}{q'_n}\right| \le c(q'_n)^{-1-\tau_n}
\end{equation}
for all $n\in\NN$. Let $p/q$ be a rational number where the denominator $q$
is large enough. We choose the minimal integer $n$ such that $2cq\le
(q'_n)^{\tau_n}$. Condition~(c) guarantees that such $n$ exists. Then, by the
choice of $n$ and Condition~(c), we have that $q\gg q_n^{\delta_n}$. By the
triangle inequality we have
$$
\left|\alpha-\frac{p}{q}\right| \ge \left|\frac{p}{q}-\frac{p'_n}{q'_n}\right| - \left|\alpha-\frac{p'_n}{q'_n}\right|.
$$
Now we have two possibilities:

\noindent{\bf (1)} The case $p/q\neq p'_n/q'_n$. Then $|p/q - p'_n/q'_n| \ge
(qq'_n)^{-1}$ and from~\eqref{prop1_eq2} we get that $|\alpha - p/q|\ge
(2qq'_n)^{-1}$. We then apply Condition~(a) to get
$$
\left|\alpha - \frac{p}{q}\right| \gg \frac{1}{q q_n^{\theta_n}} \gg \frac{1}{q^{1+\frac{\theta_n}{\delta_n}}}.
$$

\noindent{\bf (2)} The case $p/q= p'_n/q'_n$. Then we have
$$
\left|\alpha - \frac{p}{q}\right| = \left|\alpha - \frac{p'_n}{q'_n}\right|\gg (q'_n)^{-1-\tau_n}\gg q^{-(1+\tau_n)\frac{\theta_n}{\delta_n}}.
$$

To conclude the proof of the proposition, consider some number $\mu$ strictly
bigger than the right hand side in~\eqref{prop1_eq1}. Then there exists
$n_0\in \NN$ such that $\mu> 1+\theta_n/\delta_n$ and $\mu >
(1+\tau_n)\theta_n/\delta_n$ for all $n\ge n_0$. Choose $n_1\ge n_0$ such
that for $n\le n_0$ we have $(q'_n)^{\tau_n} < (q'_{n_1})^{\tau_{n_1}}$. Then
for any $p/q$ with $2cq>(q'_{n_1})^{\tau_{n_1}}$ we have that
$$
\left|\alpha - \frac{p}{q}\right| \gg q^{-\mu}
$$
and hence $\mu(\alpha)\le \mu$.
\end{proof}

\section{Irrationality exponent of $f$}

For convenience, denote the leading coefficients of $A$ and $B$ from
equation~\eqref{def_feq} by $\alpha$ and $\beta$, and denote the degrees of
$A$ and $B$ by $r_a$ and $r_b$ respectively.

Consider the sequence $(p_k(z)/q_k(z))_{k\in \ZZ_{\ge 0}}$ of the convergents
of $f$. Denote the degrees of $q_k$ by $d_k$. Then, by the standard property
of convergents, we have
\begin{equation}\label{eq_qfp}
q_k(z) f(z) - p_k(z) = \sum_{i=d_{k+1}}^\infty c_{k,i}z^{-i},
\end{equation}
where $c_{k,j}$ are some real coefficients, and $c_{k,d_{k+1}}$ is always
nonzero.

{\textsc{Proof of Theorem~\ref{th1}}.} By substituting $z^d$ instead of $z$
in equation~\eqref{eq_qfp} and then using the functional
relation~\eqref{def_feq} for $f(z^d)$ we get that:
\begin{equation}\label{eq_pqk1}
B(z)q_k(z^d) f(z) - A(z)p_k(z^d) = A(z) \sum_{i=d_{k+1}}^\infty c_{k,i}z^{-di}.
\end{equation}
By repeating this procedure $m$ times we get the following equation:
\begin{equation}\label{eq_pqkm}
q_{k,m}(z) f(z) - p_{k,m}(z) = U(z) \sum_{i=d_{k+1}}^\infty c_{k,i}z^{-d^mi},
\end{equation}
where
\begin{equation}\label{eq_defpqkm}
q_{k,m}(z) = \prod_{t=0}^{m-1} B(z^{d^t}) q_k(z^{d^m}),\quad p_{k,m}(z) =\prod_{t=0}^{m-1} A(z^{d^t}) p_k(z^{d^m}),
\quad U(z) = \prod_{t=0}^{m-1} A(z^{d^t}).
\end{equation}

\begin{lemma}\label{lem4}
Let $b\in \RR$ with $|b|>1$ be inside the disc of convergence of $f$. Assume
that for all $t\in \ZZ_{\ge 0}$, $A(b^{d^t}) B(b^{d^t})\neq 0$. Then, for $m$
large enough, one has:
\begin{equation}\label{lem4_eq}
|q_{k,m}(b)| \asymp \beta^m |b|^{d^m\left(\frac{r_b}{d-1} + d_k\right)}, \quad\mbox{and}\quad
|q_{k,m}(b) f(b) - p_{k,m}(b)| \asymp \alpha^m |b|^{d^m\left(\frac{r_a}{d-1} - d_{k+1}\right)}.
\end{equation}
Here, the constants inside the ``$\asymp$'' signs may depend on $A$, $B$, and
$k$, but do not depend on $m$.
\end{lemma}

\begin{proof}
Since $b$ is inside the disc of convergence of $f$, it is also inside the
disc of convergence of
$$
z^{d_{k+1}} (q_k(z) f(z) - p_k(z)) = \sum_{i=0}^\infty c_{k,i+d_{k+1}}
z^{-i}.
$$
By letting $z$ to infinity we see that the right hand side tends to
$c_{k,d_{k+1}}\asymp 1$. This leads us to the following expression, as $m$ is
large enough,
\begin{equation}\label{eq10}
\left|\sum_{i=d_{k+1}}^\infty c_{k,i}b^{-d^m i}\right| \asymp |b|^{-d^m d_{k+1}}.
\end{equation}

Next, notice that
$$
\prod_{t=0}^{\infty} \frac{A(z^{d^t})}{\alpha z^{d^tr_a}} = \prod_{t=0}^{\infty} P_A(z^{-d^t}),
$$
where $P_A(z)$ is a polynomial with $P(0) = 1$. One can check that the disc
of convergence for this infinite product is $|z|>1$. Moreover, since
$A(b^{d^t})\neq 0$ for all $t\in \ZZ$, we have that the product
$$
\prod_{t=0}^{m-1} \frac{A(b^{d^t})}{\alpha b^{d^tr_a}}
$$
converges to a nonzero element as $m\to \infty$. This means that,
$$
\left|\prod_{t=0}^{m-1} \frac{A(b^{d^t})}{\alpha b^{d^tr_a}}\right| \asymp 1
$$
and
\begin{equation}\label{eq11}
\left|\prod_{t=0}^{m-1} A(b^{d^t})\right| \asymp |\alpha|^m |b|^{(1+d+\ldots +d^{m-1})r_a} \asymp |\alpha|^m |b|^{\frac{d^m}{d-1}r_a}.
\end{equation}
By analogous argumenat we get the same estimate for the product over
$B(b^{d^t})$.

The last ingredient of the proof is that for $m$ large enough,
$|q_k(b^{d^m})| \asymp |b|^{d^m d_k}$. Now,
\eqref{eq_pqkm},~\eqref{eq_defpqkm},~\eqref{eq10} and~\eqref{eq11} give us:
$$
|q_{k,m}(b)| = \left|\prod_{t=0}^{m-1} B(b^{d^t}) q_k(b^{d^m})\right| \asymp |\beta|^m |b|^{d^m\left(\frac{r_b}{d-1} + d_k\right)}
$$
and
$$
|q_{k,m}(b) f(b) - p_{k,m}(b)| = \left|\prod_{t=0}^{m-1} A(b^{d^t}) \sum_{i=d_{k+1}}^\infty c_{k,i}b^{-d^mi}\right|\asymp |\alpha|^m |b|^{d^m\left(\frac{r_a}{d-1} - d_{k+1}\right)}.
$$
\end{proof}

As an immediate corollary of~\eqref{lem4_eq} we have that
\begin{equation}\label{eq_qkm_o}
|q_{k,m}(b)| = b^{d^m\left(d_k + \frac{r_b}{d-1} + o(1)\right)}
\end{equation}
and
\begin{equation}\label{eq_approx_o}
|q_{k,m}(b) f(b) - p_{k,m}(b)| = b^{-d^m\left(d_{k+1} - \frac{r_a}{d-1}
+o(1)\right)}.
\end{equation}
Since the sequence $d_k$ tends to infinity with $k$, for any $\epsilon>0$ one
can choose $k=k(\epsilon)$ such that
$$
\epsilon d_k > \max\left\{\frac{r_a}{d-1}+1, \frac{r_b}{d-1}+1\right\}.
$$
For that $k$ we can choose $m$ big enough ($m>m_0(k)$) so that the absolute
values of $o(1)$ in~\eqref{eq_qkm_o} and~\eqref{eq_approx_o} are smaller than
1/2. Then we have
\begin{equation}\label{eq_approx_eps}
q_{k,m}(b)^{-\frac{d_{k+1} (1+\epsilon)}{d_k (1-\epsilon)}} < |q_{k,m}(b) f(b) - p_{k,m}(b)| < q_{k,m}(b)^{-\frac{d_{k+1} (1-\epsilon)}{d_k (1+\epsilon)}}.
\end{equation}
Since, by letting $k\to\infty$, we can make $\epsilon$ as small as possible,
we immediately have $\mu(f(b)) \ge 1+ \limsup \frac{d_{k+1}}{d_k}$. For
convenience, let us denote that ratio $d_{k+1}/d_k$ by $\delta_k$ and define
$$
\rho:=\limsup \frac{d_{k+1}}{d_k}.
$$

\subsection{Upper bound for $\mu(f(b))$}

%For the upper bound we will use Proposition~\ref{prop1}.

For a given $k_0\in\NN$ define $K = K(k_0)\in\NN$ as the minimal possible
value such that
\begin{equation}\label{ineq_bigk}
d_{K+1}> d\cdot d_{k_0+1}-r_a+d+1.
\end{equation}
Consider any $k$ in the range $k_0\le k<K$ and consider an arbitrary $m>
M=M(k_0):=\max\{m_0(k_0), m_0(k_0+1), \ldots, m_0(K)\}$, so that the
equation~\eqref{eq_approx_eps} is satisfied for all values $k$ between
$k_0$ and $K$. Equation~\eqref{eq_approx_eps} yields to
\begin{equation}\label{eq_epskm}
\left|f(b) -\frac{p_{k,m}(b)}{q_{k,m}(b)}\right| \asymp q_{k,m}^{ -1 - \left(\frac{d_{k+1}}{d_k} +\epsilon_{k,m}\right)},
\end{equation}
where $\sup_{k_0\le k<K,\; m\ge M} |\epsilon_{k,m}|$ tends to zero as $k_0$
tends to infinity.

Now we construct sequences $P_n/Q_n$ and $P'_n/Q'_n$ for
Proposition~\ref{prop1} in the following way:
$$
\frac{P_1}{Q_1} := \frac{p_{k_0,M}(b)}{q_{k_0,M}(b)},\; \frac{P_2}{Q_2} := \frac{p_{k_0+1,M}(b)}{q_{k_0+1,M}(b)},\ldots,\; \frac{P_{K-k_0}}{Q_{K-k_0}} := \frac{p_{K-1,M}(b)}{q_{K-1,M}(b)};
$$
$$
\frac{P'_1}{Q'_1} := \frac{p_{k_0+1,M}(b)}{q_{k_0+1,M}(b)},\; \frac{P'_2}{Q'_2} := \frac{p_{k_0+2,M}(b)}{q_{k_0+2,M}(b)},\ldots,\; \frac{P'_{K-k_0}}{Q'_{K-k_0}} := \frac{p_{K,M}(b)}{q_{K,M}(b)}.
$$
Then we continue defining the sequences by increasing the index $M$. That is,
for any $u\in\ZZ_{\ge0}$ and any $v\in \{1,\ldots, K-k_0\}$ we define
$$
\frac{P_{u(K-k_0)+v}}{Q_{u(K-k_0)+v}} := \frac{p_{k_0+v-1,M+u}(b)}{q_{k_0+v-1,M+u}(b)};\quad  \frac{P'_{u(K-k_0)+v}}{Q'_{u(K-k_0)+v}} := \frac{p_{k_0+v,M+u}(b)}{q_{k_0+v,M+u}(b)}\ .
$$
From~\eqref{eq_epskm} one can see that the following sequences
$(\delta_n)_{n\in\NN}$ and $(\tau_n)_{n\in\NN}$ satisfy Condition~(b) of
Proposition~\ref{prop1}:
$$
\delta_{u(K-k_0)+v}:= \frac{d_{k_0+v}}{d_{k_0+v-1}} +\epsilon_{k_0+v-1, M+u};\quad \tau_{u(K-k_0)+v}:= \frac{d_{k_0+v+1}}{d_{k_0+v}} +\epsilon_{k_0+v, M+u}.
$$
%Moreover, we define $\delta_n$ and $\tau_n$ in such a way that the exact
%equalities are satisfied in Condition~(b).

Let us define a sequence $(\theta_n)_{n\in\NN}$ so that Condition~(a) is
satisfied. By~\eqref{eq_qkm_o} we have that for any $k\in\{k_0,\ldots, K\}$
and for any $m>M$,
$$
|q_{k+1,m} (b)| = |q_{k,m}(b)|^{\frac{d_{k+1}+r_b/(d-1)+o(1)}{d_k+r_b/(d-1)+o(1)}} = |q_{k,m}(b)| ^{\frac{d_{k+1}}{d_k}+\epsilon^*_{k,m}},
$$
where, as for $\epsilon_{k,m}$, $\sup_{k_0\le k<K,\; m\ge M}
|\epsilon^*_{k,m}|\to 0$ as $k_0$ tends to infinity. The last equation
suggests the following formula for $\theta_n$:
$$
\theta_{u(K-k_0)+v} := \frac{d_{k_0+v}}{d_{k_0+v-1}} +\epsilon^*_{k_0+v-1, M+u}.
$$

It remains to verify Condition~(c). The limit $Q_n^{\delta_n}\to\infty$ is
obvious. Because of Condition~(b), the equation $(Q'_n)^{\tau_n} \gg
Q_{n+1}^{\delta_{n+1}}$ is equivalent to:
\begin{equation}\label{cond_qqdash}
|Q'_n f(b) - P'_n| \ll |Q_{n+1} f(b) - P_{n+1}|.
\end{equation}
%The last inequality is verified by the estimate~\eqref{eq_approx_o}, the
%inequality $d_{k+2}\ge d_{k+1}+1$ and $o(1)<1/2$. Indeed, for any $k\ge k_0$
%and $m>M$ we have
%$$
%\begin{aligned}
%|q_{k+1,m}(b)f(b) - p_{k+1,m}(b)| &=
%b^{-d^m\left(d_{k+2}-\frac{r_a}{d-1}+o(1)\right)} \\
%&<b^{-d^m\left(d_{k+1}-\frac{r_a}{d-1}+o(1)\right)} = |q_{k,m}(b)f(b) -
%p_{k,m}(b)|.
%\end{aligned}
%$$
By definition, we have that for any $u\in \ZZ_{\ge 0}$ and $v\in \{1,\ldots
K-k_0-1\}$, $Q'_{u(K-k_0)+v} = Q_{u(K-k_0)+v+1}$ and $P'_{u(K-k_0)+v} =
P_{u(K-k_0)+v+1}$ and both sides of~\eqref{cond_qqdash} coincide for
$n=u(K-k_0)+v$. Therefore it only remains to verify~\eqref{cond_qqdash} for
$n=(u+1)(K-k_0)$. From the estimate~\eqref{eq_approx_o} and equations
$Q'_{(u+1)(K-k_0)} = q_{K,M+u}(b)$, $Q_{(u+1)(K-k_0)+1} = q_{k_0, M+u+1}(b)$
we have
$$
\begin{aligned}
|q_{K,M+u}(b)f(b) - p_{K,M+u}(b)| &=
b^{-d^{m+u}\left(d_{K+1}-\frac{r_a}{d-1}+o(1)\right)} \\
\stackrel{\eqref{ineq_bigk}}<b^{-d^{m+u+1}\left(d_{k_0+1}-\frac{r_a}{d-1}+o(1)\right)} &= |q_{k_0,M+u+1}(b)f(b) -
p_{k_0,M+u+1}(b)|.
\end{aligned}
$$

After all conditions of Proposition~\ref{prop1} are checked, we apply it to
get
$$
\mu(f(b)) \le \limsup_{u\to\infty} \max_{k_0\le v<K} \left\{1+ \frac{\theta_{u(K-k_0)+v}}{\delta_{u(K-k_0)+v}}, (1+\tau_{u(K-k_0)+v})\frac{\theta_{u(K-k_0)+v}}{\delta_{u(K-k_0)+v}}\right\}.
$$
Notice that by construction, $\theta_n/\delta_n$ tends to one as $k_0$ tends
to infinity. Also, as $k_0$ tends to infinity, we have that
$$
\tau_{u(K-k_0)+v} \to \frac{d_{k_0+v+1}}{d_{k_0+v}}.
$$
for all $u\ge 0$ and $v$ between $k_0$ and $K$. This leads to the upper bound
$$
\mu(f(b))\le \limsup_{n\to\infty} \left\{1 + \frac{d_{n+1}}{d_n}\right\},
$$
which now coincides with the lower bound for $\mu(f(b))$. That proves
Theorem~\ref{th1}. \endproof

\section{Gaps in the set of values $d_k$}

Theorem~\ref{th1} suggests that in order to compute the irrationality
exponent of a Mahler number $f(b)$, we need to consider large gaps in the
sequence $(d_k)_{k\in \NN}$ of degrees of the denominators of the convergents
of $f(z)$.

Define by $\Phi$ the set of all values $d_k$:
$$
\Phi = \Phi(f) := \{d_k\;:\; k\in \NN\}.
$$
We say that $[u,v]$ is {\it a gap in $\Phi$ of size $r>0$} if $u$ and $v$ are
elements of $\Phi$, $r = v-u$ and no elements $w$ with $u<w<v$ are in $\Phi$.
For the gap $[u,v]$ in $\Phi$ we say that $p_u(z)/q_u(z)$ is {\it gap's
convergent} if $p_u(z)/q_u(z)$ is a convergent of $f$ and $\deg(q_u) = u$.

In further discussion we always assume that the value $b\in \NN$ satisfies
the conditions of Theorem~\ref{th1}. It implies that if all gaps in $\Phi$
are of size at most $\frac{r_a+r_b}{d-1}$ then $\mu(f(b))=2$. Indeed, we have
$$
\mu(f(b)) = 1+ \limsup_{\mbox{\scriptsize gaps }[u,v]\mbox{\scriptsize \;of }\Phi\;} \frac{v}{u}\le \lim_{u\to\infty}1+\frac{u + \frac{r_a+r_b}{d-1}}{u} = 2.
$$
Therefore in order to compute the irrationality exponent of $f(b)$ it is
sufficient to consider gaps in $\Phi$ of a bigger size than
$\frac{r_a+r_b}{d-1}$. We call such gaps big. We introduce a partial order on
the set of big gaps. We say that $[u,v]\prec [u',v']$ if there exists
$m\in\NN$ such that
$$
\frac{p_{u'}(z)}{q_{u'}(z)} = \prod_{t=0}^{m-1} \frac{A(z^{d^t})}{B(z^{d^t})} \cdot \frac{p_u(z^{d^m})}{q_u(z^{d^m})}.
$$
This definition is justified by the following lemma.
\begin{lemma}
Let $[u,v]$ be a big gap in $\Phi$. Then the fraction
$$
\frac{A(z)p_u(z^d)}{B(z)q_u(z^d)}
$$
is a convergent of $f$. Moreover, the gap in $\Phi$, which corresponds to
this convergent, has size bigger than $v-u$.
\end{lemma}
\begin{proof}
Denote by $C(z)$ the polynomial $\gcd(A(z)p_u(z^d), B(z)q_u(z^d))$ and let
$r_c:= \deg(C)$. From~\eqref{eq_pqk1} we have that
\begin{equation}\label{eq_approx_pqu}
\left|\left|\frac{B(z) q_u(z^d)}{C(z)} f(z) - \frac{A(z) p_u(z^d)}{C(z)}\right|\right| = r_a - r_c - dv.
\end{equation}
Here, $||g||$ denotes the biggest degree of $z$ with non-zero coefficient in
Laurent series $g$. We have that $\frac{B(z) q_u(z^d)}{C(z)}$ and $\frac{A(z)
p_u(z^d)}{C(z)}$ are coprime and moreover,
\begin{equation}\label{eq_qnp1}
\deg\left( \frac{B(z) q_u(z^d)}{C(z)} \right) = r_b + du-r_c < dv + r_c -r_a.
\end{equation}
The last inequality is true because for big gaps we have $v-u>
\frac{r_a+r_b}{d-1}$. Hence
$$
\frac{A(z) p_u(z^d)}{C(z)} \big/ \frac{B(z) q_u(z^d)}{C(z)}
$$
is a convergent of $f$ and the size of its corresponding gap is
$$
(dv+r_c-r_a) - (r_b+du-r_c) = 2r_c + d(v-u) - r_a - r_b > v-u.
$$
\end{proof}

We say that a big gap $[u,v]$ in $\Phi$ is {\it primitive} if there are no
other big gaps $[u',v']$ in $\Phi$ such that $[u',v']\prec [u,v]$. A
primitive gap $[u,v]$ generates the ordered sequence of big gaps
$$
[u,v] = [u_0,v_0] \prec [u_1,v_1]\prec [u_2,v_2]\prec\cdots
$$
such that
\begin{equation}\label{eq21}
\frac{p_{u_{n+1}}(z)}{q_{u_{n+1}}(z)} = \frac{A(z) p_{u_n}(z^d)}{C(z)} \big/ \frac{B(z) q_{u_n}(z^d)}{C(z)}.
\end{equation}
Then the formula~\eqref{th1_eq} for $\mu(f(b))$ from Theorem~\ref{th1} can be
rewritten as follows:
\begin{equation}\label{eq_mufb_prim}
\mu(f(b)) = 1+\sup_{[u_0,v_0]\mbox{\scriptsize\ is primitive}} \left\{ \limsup_{i\to\infty}\frac{v_{n_i}}{u_{n_i}}\right\} \cup\{1\}.
\end{equation}

\begin{lemma}\label{lem6}
The size of a primitive gap in $\Phi$ does not exceed
$\frac{2d-1}{d-1}(r_a+r_b)$.
\end{lemma}

\begin{proof}
Suppose the contrary: the size of a primitive gap $[u,v]$ in $\Phi$ is bigger
than $\frac{2d-1}{d-1}(r_a+r_b)$. Let $w$ be the biggest integer such that
$dw<v-r_b$ (i.e. $w = \lfloor\frac{v-r_b-1}{d}\rfloor$).

Assume that $w$ lies inside a big gap $[s,t]$ in $\Phi$, that is, $s\le w<t$.
Then, by~\eqref{eq_approx_pqu} and~\eqref{eq_qnp1} the gap, associated with
the convergent
$$
\frac{A(z)p_s(z^d)}{B(z)q_s(z^d)},
$$
contains $[ds+r_b, dt-r_a]$. Obviously, $ds+r_b<v$ and $dt-r_a\ge v-r_b-r_a
>u$. Therefore this gap intersects with $[u,v]$ and hence it must coincide with $[u,v]$.
We get $[s,t]\prec[u,v]$, which is a contradiction.

We then deduct that $w$ does not lie inside a big gap. In other words, there
is an element $s\in\Phi$ with $0\le w-s\le \frac{r_a+r_b}{d-1}$. Consider the
fraction
$$
\frac{p(z)}{q(z)} = \frac{A(z)p_s(z^d)}{B(z)q_s(z^d)}.
$$
Then by~\eqref{eq_pqk1}, we have $\|q(z) f(z) - p(z)\| \le r_a - d(s+1)$,
which is strictly smaller than $-u$. Indeed,
$$
d(s+1)-r_a\ge d\left( w- \frac{r_a+r_b}{d-1} +1\right) -r_a > v-r_a-r_b -\frac{d}{d-1}(r_a+r_b)\ge u.
$$

Divide $q$ by $q_u$ with the remainder: $q(z) = a(z) q_u(z) + r(z)$ and write
$p(z) = a(z) p_u(z) + c(z)$. Then we have
$$
||a(z) q_u(z) f(z) - a(z) p_u(z) || = \deg(a) -v.
$$
Obviously, the degree of $q$ is $r_b+ds$ which is strictly smaller than $v$
and therefore $\deg(a) - v < v-u - v = -u$.

Assume that $r\neq 0$. Since the convergents of $f$ are the best approximants
to $f$ and $\deg(r)<\deg(q_u)$, we have
$$
||r(z)f(z) - c(z)|| \ge ||q_{u'} (z) f(z) - p_{u'}(z)|| = -u,
$$
where $p_{u'}/q_{u'}$ is the convergent of $f$ which precedes $p_u/q_u$. The
last two estimates imply
$$
||q(z) f(z) - p(z) || = ||r(z)f(z) - c(z)|| \ge -u,
$$
which contradicts the condition $r\neq 0$. Hence we get that $r = 0$ and
$p/q$ coincides with $p_u/q_u$. This together with~\eqref{eq_approx_pqu}
and~\eqref{eq_qnp1} implies that $u = ds+r_b-r_c$ and $v = dt+r_c-r_a$ where
$r_c = \deg(\gcd(A(z)p_u(z^d), B(z)q_u(z^d)))$. Since polynomials $p_u(z^d)$
and $q_u(z^d)$ are coprime, $r_c\le r_a+r_b$. Finally,
$$
v-u \le d(t-s)+2r_c - r_a-r_b \le \frac{2d-1}{d-1} (r_a+r_b).
$$

%The degree of its denominator $r_b + ds$ is strictly smaller than $v$. Thus,
%since convergents $p_n(z)/q_n(z)$ of $f(z)$ are the best approximations of
%$f(z)$, we have
%$$
%\| q(z) f(z) - p(z)\| \ge \|q_u(z)f(z) - p_u(z)\| = -v.
%$$
%But by~\eqref{eq_pqk1}, $\|q(z) f(z) - p(z)\| \le r_a - d(s+1)$, which is
%strictly smaller than $r_a-v$. Indeed,
%$$
%d(s+1)\ge d\left( w- \frac{r_a+r_b}{d-1} +1\right) > v-r_b-d -\frac{d}{d-1}(r_a+r_b) + d\ge u.
%$$
\end{proof}

\begin{lemma}\label{lem7}
Assume that $(d-1)u>r_a$. Then
\begin{equation}\label{lem7_ineq}
\frac{v - \frac{r_a}{d-1}}{u + \frac{r_b}{d-1}}\le \limsup_{n\to\infty} \left\{\frac{v_n}{u_n}\right\} \le \frac{v + \frac{r_b}{d-1}}{u - \frac{r_a}{d-1}}.
\end{equation}
\end{lemma}

\begin{proof}
From equations~\eqref{eq_approx_pqu} and~\eqref{eq_qnp1} we have that
\begin{equation}\label{form_un}
\begin{array}{l}
u_{n+1} = du_n+r_b - r_{c,n};\\
v_{n+1} = dv_n-r_a+r_{c,n},
\end{array}
\end{equation}
where
\begin{equation}\label{def_rcn}
r_{c,n} = \deg\gcd(A(z) p_{u_n}(z^d), B(z) q_{u_n}(z^d))
\end{equation}
and it is not bigger than $r_a+r_b$. This implies
$$
\frac{dv_n-r_a}{du_n+r_b}\le\frac{v_{n+1}}{u_{n+1}}\le \frac{dv_n+r_b}{du_n-r_a}.
$$
By iterating this inequality $n$ times we get
$$
\frac{d^nv-(1+d+\ldots +d^{n-1})r_a}{d^nu+(1+d+\ldots +d^{n-1})r_b} \le \frac{v_n}{u_n}\le \frac{d^nv+(1+d+\ldots +d^{n-1})r_b}{d^nu-(1+d+\ldots +d^{n-1})r_a}.
$$
Taking limits as $n\to\infty$ yields~\eqref{lem7_ineq}.
\end{proof}

Lemmata~\ref{lem6} and~\ref{lem7} together imply that only finitely many
primitive gaps may contribute to the supremum in~\eqref{eq_mufb_prim}.
Indeed, consider all primitive gaps $[u,v]$ in $\Phi$ with
$u>\frac{r_a}{d-1}$. By Lemma~\ref{lem6}, their sizes are bounded. Therefore
we can choose the primitive gap $[u_0,v_0]$ in $\Phi$ with the biggest
possible size $S$ such that $u_0$ is smallest possible among all primitive
gaps in $\Phi$ of this size. Then, by Lemma~\ref{lem7}, a primitive gap
$[u,v]$ in $\Phi$ can only contribute to the limsup in~\eqref{eq_mufb_prim}
if
$$
\frac{u+S+\frac{r_b}{d-1}}{u-\frac{r_a}{d-1}} > \frac{v_0-\frac{r_a}{d-1}}{u_0+\frac{r_b}{d-1}}.
$$
Since $v_0-u_0>\frac{r_a+r_b}{d-1}$, the right hand side of the inequality is
bigger than one and therefore it gives us an upper bound for $u$. Denote this
bound by $l_u$. We deduct that only $[u,v]$ with $u\le l_u$ can contribute to
the limsup in~\eqref{eq_mufb_prim} and there are obviously finitely many of
them.

Now to complete the proof of Theorem~\ref{th2}, we need to show that for any
primitive gap $[u,v]$ in $\Phi$ we have that $\liminf_{n\to\infty} v_n/u_n$
is a rational number. The most misterious term in the
formulae~\eqref{form_un} is $r_{c,n}$. In the next section we prove the
following proposition which is a key to the proof of Theorem~\ref{th2}.

\begin{proposition}\label{prop2}
The sequence $(r_{c,n})_{n\in\NN}$ is eventually periodic.
\end{proposition}

We end this section by showing how Proposition~\ref{prop2} implies
Theorem~\ref{th2}. Let the sequence $(r_{c,n})_{n\in\NN}$ be periodic,
starting from the index $n_0$ and with the period length $P$, i.e.
$r_{c,n_0+i} = r_{n_0+P+i}$ for every $i\in\ZZ_{\ge 0}$. Denote by $R$ the
following value:
$$
R:= d^{P-1}r_{c,n_0}+\ldots + dr_{c,n_0+P-2}+r_{c,n_0+P-1}.
$$

By applying the formulae~\eqref{form_un} for $u_{n_0}, v_{n_0}, u_{n_0+1},
v_{n_0+1},$... up to $u_{n_0+P}, v_{n_0+P}$, we get
$$
u_{n_0+P} = d^Pu_{n_0} + r_b(1+d+\cdots+d^{P-1})-R,\;\; v_{n_0+P} = d^Pv_{n_0} -r_a(1+d+\cdots+d^{P-1})+R.
$$
Define
$$
r_u:= r_b(1+d+\cdots+d^{P-1})-R,\quad\mbox{and}\quad r_v:=r_a(1+d+\cdots+d^{P-1})-R.
$$
Then we get
$$
\lim_{k\to\infty} \frac{u_{n_0+kP}}{v_{n_0+kP}} = \lim_{k\to\infty}\frac{d^{kP}u_{n_0} + (1+d^P+d^{2P}+d^{(k-1)P})r_u}{d^{kP}v_{n_0} - (1+d^P+d^{2P}+d^{(k-1)P})r_v} = \frac{u_{n_0} + \frac{r_u}{d^P-1}}{v_{n_0} - \frac{r_v}{d^P-1}},
$$
which is a rational number. By analogous arguments, the limits of
$u_{n_0+1+kP}/v_{n_0+1+kP}, \ldots$, $u_{n_0+(k+1)P-1}/v_{n_0+(k+1)P-1}$ as
$k\to\infty$ are all rational numbers. Therefore $\limsup_{n\to\infty}
v_n/u_n$, as the maximum of the limits above, is a rational number. This
finishes the proof of Theorem~\ref{th2}.

\section{Proof of Proposition~\ref{prop2}}

We split each of the polynomials $A,B,p_{u_m}$ and $q_{u_m}$ ($m\in\ZZ_{\ge
0}$) into the product of three factors: cyclotomic, non-cyclotomic and the
power of $z$. For example, $A(z) = A_c(z)\cdot A_n(z)\cdot A_0(z)$, where all
roots of $A_c(z)$ are roots of unity, $A_0(z)$ is a power of $z$ and none of
the roots of $A_n(z)$ is either zero or a root of unity. The polynomials
$B_c, B_n, B_0, p_{c,u_m}, p_{n,u_m}, p_{0,u_m}, q_{c,u_m}, q_{n,u_m}$ and
$q_{0,u_m}$ are defined in the same way. Obviously,
$$
\gcd(A_c(z)p_{c,u_m}(z^d), B_n(z) q_{n,u_m}(z^d)) = \gcd(A_n(z)p_{n,u_m}(z^d), B_c(z) q_{c,u_m}(z^d)) = \mathrm{const},
$$
and therefore we can split $r_{c,m}$ into the sum of three parts: $r_{c,m} =
r_{c,c,m}+r_{n,c,m}+r_{0,c,m}$. The first term is the degree of the
cyclotomic part of the gcd in~\eqref{def_rcn}, the second one is the degree
of the non-cyclotomic part of it and the third one is generated by the powers
of $z$ presented in the $\gcd$. We will consider each term separately.

\subsection{Non-cyclotomic term}

\begin{lemma}\label{lem8}
Let $C,D\in \ZZ[x]$ be such that none of their roots is a root of unity. Then
there exists $m_0\in\NN$ such that for all $m>m_0$,
$$
\gcd(C(z), D(z^{d^m})) = \mathrm{const}.
$$
\end{lemma}

\begin{proof}
Assume the contrary. Then there exists a root $\alpha$ of $C$ such that
$z-\alpha$ divides $D(z^{d^m})$ for infinitely many values $m$. Hence there
exists a root $\beta$ of $D$ such that
$$
\beta = \alpha^{d^{m_1}} = \alpha^{d^{m_2}}
$$
for some positive integers $m_1\neq m_2$. But the latter is only possible if
$\alpha$ is a root of unity or zero --- a contradiction.
\end{proof}

\begin{lemma}\label{lem99}
The sequence $(r_{n,c,m})_{m\in\NN}$ is eventually periodic.
\end{lemma}

\begin{proof}
From Lemma~\ref{lem8} fix $m_0$ so that
$$
\gcd(A_n(z), B_n(z^{d^m})q_{n,u}(z^{d^m})) = \gcd(B_n(z), A_n(z^{d^m})p_{n,u}(z^{d^m})) = \mathrm{const}
$$
for all $m\ge m_0$. Write the non-cyclotomic part of the convergent
$p_{u_m}/q_{u_m}$ in the following form:
$$
\frac{p_{n,u_m}}{q_{n,u_m}} = \frac{\prod_{t=0}^{m-1} A^*_{t,m}(z) p^*_m(z)}{\prod_{t=0}^{m-1}B^*_{t,m}(z) q^*_m(z)},
$$
where the numerator and denominator of the right hand side are coprime;
$A^*_{t,m}(z) \mid A(z^{d^t})$, $B^*_{t,m}(z)\mid B(z^{d^t})$, $p^*_m(z) \mid
p_n(z^{d^m})$ and $q^*_m(z) \mid q_n(z^{d^m})$; the leading coefficients of
$A_{t,m}^*(z), B^*_{t,m}(z), p_m^*(z), q_m^*(z)$ coincide with those of
$A(z), B(z), p_{n}(z)$ and $q_{n}(z)$ respectively. Then for $m>m_0$ the
degree $r_{n,c,m}$ of
$$
\gcd(A_n(z) p_{n,u_{m+1}}(z^d), B_n(z) q_{n,u_{m+1}}(z^d))
$$
as well as the polynomials $A^*_{t,m+1}, B^*_{t,m+1}$, $t\in\{0,\ldots,
m_0\}$, depend entirely on the polynomials $A^*_{0,m}$, $A^*_{1,m}, \ldots$,
$A^*_{m_0,m}$ and $B^*_{0,m}, \ldots, B^*_{m_0,m}$. But there are finitely
many such combinations. Therefore one can find $m_2>m_1>m_0$ such that
$A^*_{0,m_1} = A^*_{0,m_2}$, \ldots, $A^*_{m_0,m_1} = A^*_{m_0,m_2}$ and
$B^*_{0,m_1} = B^*_{0,m_2}$, \ldots, $B^*_{m_0,m_1} = B^*_{m_0,m_2}$. Then we
get $r_{n,c,m_1} = r_{n,c,m_2}, r_{n,c,m_1+1} = r_{n,c,m_2+1}$, etc. Hence
the sequence of $(r_{c,n,m})_{m\in \NN}$ is eventually periodic.
\end{proof}

\subsection{Powers of $z$}

We write
$$
A_0(z) = z^{s_a};\quad B_0(z)=z^{s_b};\quad p_{0,u}(z) = z^{s_p}\quad\mbox{and}\quad q_{0,u}(z) = z^{s_q}.
$$
Since $A$ and $B$ are coprime and $p_u$ and $q_u$ are coprime, we have that
at least one value of $s_a, s_b$ and at least one of $s_p, s_q$ is zero. If
we have $s_a = s_p=0$ (or $s_b=s_q=0$) then, by~\eqref{eq21}, the powers of
$z$ of all numerators (denominators) are zero and hence
$(r_{0,c,m})_{m\in\ZZ_{ge 0}}$ is the zero sequence.

Now without loss of generality assume that $s_a>0, s_q>0, s_b=s_p=0$. Denote
by $s_{p,m}$ and $s_{q,m}$ the maximal powers of $z$ of $p_{u_m}$ and
$q_{u_m}$ respectively. Notice that, if for some $m_0\in\NN$ the value
$s_{q,m_0}$ is zero then, as before, the sequence $r_{0,c,m}$ becomes zero
for all $m\ge m_0$. On the other hand, if $s_{q,m}$ is positive for all
$m\in\ZZ_{\ge 0}$ then the power of $z$ of $q_{u_m}(z^d)$ is always bigger
than that of $A(z)$, which follows that $r_{0,c,m}$ equals $s_a$ for all
$m\in\ZZ_{\ge 0}$.

In all cases we have that the sequence $(r_{0,c,m})_{m\in\ZZ_{\ge 0}}$ is
eventually periodic.

\subsection{Cyclotomic term}

Note that each of the polynomials $A_c, B_c, p_{c,u}, q_{c,u}$ is a (possibly
empty) product of cyclotomic polynomials $\Phi_n(z)$. We start by
investigating the structure of polynomials $\Phi_n(z^d)$ as $d$ changes. That
requires some notation. Given $n\in \NN$, the radical of $n$ is the product
of all prime divisors of $n$, i.e.:
$$
\rad(n):= \prod_{{p\in \PP}\atop{p\mid n}} p.
$$
For two positive integers $n$ and $m$, by $r(n,m)$ we denote the biggest
divisor of $n$ which is coprime with $m$, and $s(n,m):=n/r(n,m)$.

\begin{lemma}\label{lem9}
Let $n,d$ be two positive integers. The polynomial $\Phi_n(z^d)$ is a product
of cyclotomic polynomials. More precisely,
$$
\Phi_n(z^d) = \prod_{r\mid r(d,n)} \Phi_{rns(d,n)}(z).
$$
\end{lemma}

\begin{proof}
All the roots of $\Phi_n(z)$ are of the form $\xi_n^i$, where $\xi_n$ is
$n$-th primitive root of unity, $0\le i<n$ and $\gcd(i,n)=1$. Therefore the
roots $\xi$ of $\Phi_n(z^d)$ are the solutions of the equation $\xi^d =
\xi_n^i$, which can be written as
$$
\xi_{nd}^i \cdot \xi_d^j = \xi_{nd}^{nj+i},
$$
where $0\le j<d$. The values $nj+i$ run through the set $\NNN$ of all numbers
between zero and $nd$, which are coprime with $n$. Split this set into
subsets
$$
\NNN_t:=\{x\in\NNN\;:\; \gcd(d,x)=t\}.
$$
Obviously, they are non-empty only if $t\mid d$ and $\gcd(t,n)=1$. These two
conditions are equivalent to $t\mid r(d,n)$. Denote by $r$ the fraction
$r(d,n)/t$. Notice that for any $x\in\NNN_t$ one has $\xi_{nd}^x =
\xi_{nd/t}^{x/t}$ where $x/t$ is coprime with $nd/t$. Finally, write $nd/t =
rns(d,n)$, so the numbers $\xi_{nd}^x$ are the roots of the polynomial
$\Phi_{rns(d,n)}(z)$ and $\Phi_{rns(d,n)}(z)\mid \Phi_n(z^d)$.
\end{proof}

Write the polynomial $A_c(z)$ as the product:
$$
A_c(z) = \prod_{{r\in\NN}\atop{\gcd(d,r)=1}} A_{r,c}(z),
$$
where $A_{r,c}(z)$ is the product of all $\Phi_n(z)$ such that $\Phi_n(z)\mid
A_c(z)$ and $r(n,d)=r$. Other polynomials $B_{r,c}(z), p_{r,c,u_m}(z),
q_{r,c,u_m}(z)$ are defined analogously. Clearly, among all values $r$ with
$\gcd(r,d)=1$ only finitely many polynomials $A_{r,c}(z)$ have positive
degree.

One of the outcomes of Lemma~\ref{lem9} is that for any $n$ and $m$ in $\NN$
every cyclotomic divisor $\Phi_k(z)$ of $\Phi_n(z^{d^m})$ has $r(k,d) =
r(n,d)$. Therefore we can split $r_{c,c,m}$ into the sum:
$$
r_{c,c,m} = \sum_{{r\in\NN}\atop{\gcd(d,r)=1}} r_{r,c,c,m},
$$
where
$$
r_{r,c,c,m} = \gcd(A_{r,c}(z)p_{r,c,u_m}(z^d), B_{r,c}(z)q_{r,c,u_m}(z^d)).
$$
Only finitely many of the sequences $(r_{r,c,c,m})_{m\in\NN}$ are non-zero.

It remains to show that every non-zero sequence $(r_{r,c,c,m})_{m\in\NN}$ is
eventually periodic.

{\bf Case 1.} Assume that among the divisors of $A_{r,c}, B_{r,c}, p_{r,c,u},
q_{r,c,u}$ there are no polynomials $\Phi_r(z)$. From Lemma~\ref{lem9} we
know that all divisors $\Phi_k(z)$ of $\Phi_n(z^{d^m})$ satisfy $n
s(d,n)^m\mid k$. Consider a divisor $\Phi_n(z)$ of one of the polynomials
$A_{r,c}, B_{r,c}, p_{r,c,u}, q_{r,c,u}$. Since $n\neq r$, and $r(n,d)=r$, we
have $s(d,n)>1$ and therefore, as $m$ tends to infinity, all divisors
$\Phi_k(z)$ of $\Phi_n(z^{d^m})$ satisfy $k\to\infty$. Therefore there exists
$m_0$ such that for $m>m_0$
$$
\gcd(A_{r,c}(z), \Phi_n(z^{d^m})) = \gcd(B_{r,c}(z), \Phi_n(z^{d^m})) = \mathrm{const}.
$$
Then the proof of Proposition in this case is analogous to that of
Lemma~\ref{lem99}.

Before considering the other cases, we need more notation and lemma. Given
two polynomials $f(z),g(z)\in\ZZ[z]$ with $\deg(f)>0$ denote by $\sigma(f,g)$
the maximal power of $f$ which divides $g$, i.e.
$$
\sigma(f,g):=\max\{n\in\ZZ_{\ge 0}\;:\; (f(z))^n\mid g(z)\}.
$$

\begin{lemma}\label{lem11}
For any $f(z)\in\ZZ[z]$ and any $k\in \NN$ there exists a constant $c =
c(f,k)$ such that for any $m\in\NN$, $\sigma(\Phi_k(z), f(z^{d^m})) <c$.
\end{lemma}

\begin{proof}
We write $k = rs$ where $r = r(k,d)$ is coprime with $d$ and the radical of
$s$ divides the one of $d$. Split $f$ as a product $f = f_r\cdot g$, where
all the roots of $g$ are either not roots of unity or they are roots of unity
of degree $k'$ with $r(k',d)\neq r$. The function $f_r$ is defined as
follows:
$$
f_r(z) = \prod_{\rad(s_i)\,\mid\, \rad(d)} (\Phi_{rs_i}(z))^{\alpha(s_i)}.
$$
Then $g(z^{d^m})$ is always coprime with $\Phi_k(z)$ and $\sigma(\Phi_k(z),
f(z^{d^m})) = \sigma(\Phi_k(z), f_r(z^{d^m}))$.

We finish the proof of the lemma by  induction on $s$. For $s=1$,
Lemma~\ref{lem9} implies that $$\sigma(\Phi_r(z), f_r(z^{d^m})) =
\sigma(\Phi_r(z), (\Phi_r(z^{d^m}))^{\alpha(1)}) = \alpha(1).$$ Now, consider
$S\in \NN$ with $\rad(S)\mid \rad(d)$. Assume that the statement of the lemma
is satisfied for all $s<S$ with $\rad(s)\mid d$, i.e. for any such $s$ there
exists a constant $c(s)$ such that $\sigma(\Phi_{rs}(z), f_r(z^{d^m}))\le
c(s)$. Now we prove the statement for $S$. Lemma~\ref{lem9} implies that
$$
\sigma(\Phi_{rS}(z), f_r(s^{d^m}))\le \sigma(\Phi_{rS}(z), (\Phi_{rS}(z^{d^m}))^{\alpha(S)}) + \!\!\!\!\sum_{s\mid S,\; s<S}\!\!\!\! \sigma(\Phi_{rs}(z),f_r(z^{d^{m-1}})) \le \alpha(S)+\!\!\!\!\sum_{s\mid S,\; s<S}\!\!\! c(s).
$$
Since the right hand side does not depend on $m$, the proof is finished.
\end{proof}

{\bf Case 2.} Assume that $\Phi_r(z)$ divides $p_{r,c,u}(z)$ and
$$\gcd(\Phi_r(z), A_{r,c}(z)) = \gcd(\Phi_r(z), B_{r,c}(z)) =
\mathrm{const}.$$

Note that the case $\Phi_r(z)\mid q_{r,c,u}(z)$ can be dealt analogously: we
just swap $A_{r,c}$ with $B_{r,c}$ and $p_{r,c,u}$ with $q_{r,c,u}$.

Write $A_{r,c}(z)$ and $B_{r,c}(z)$ as
$$
A_{r,c}(z) = \prod_{i=1}^n \Phi_{rs_i}(z),\quad B_{r,c}(z) = \prod_{i=n+1}^{n+n^*} \Phi_{rs_i}(z).
$$
Let $\SSS$ be the set of all positive integers $s$ which divide one of the
values $s_i$, $1\le i\le n+n^*$, i.e.
$$
\SSS:=\{s\in\NN\;:\; \exists i\in\{1,\ldots, n+n^*\},\; s\mid s_i\}.
$$

Recall that $p_{r,c,u_{m+1}}/q_{r,c,u_{m+1}}$ can be written in the form
$$
\frac{p_{r,c,u_{m+1}}(z)}{q_{r,c,u_{m+1}}(z)} = \frac{A_{r,c}(z)p_{r,c,u_m}(z^d)}{B_{r,c}(z)q_{r,c,u_m}(z^d)}.
%\frac{\prod_{t=0}^{m-1} A^*_{t,m}(z) p^*_m(z)}{\prod_{t=0}^{m-1}B^*_{t,m}(z) q^*_m(z)},
$$
Then the value $r_{r,c,c,m+1}$ is completely determined by two tuples
$\Sigma_{p,m}$ and $\Sigma_{q,m}$ which are defined as follows:
$$
\Sigma_{p,m}:=(\sigma_{\zeta_1,m}, \ldots, \sigma_{\zeta_N,m}),\quad\mbox{where}\; N=\#\SSS,\; \zeta_1,\ldots, \zeta_n\in\SSS,\;\sigma_{\zeta,m}:=\sigma(\Phi_{r\zeta}(z),p_{r,c,u_m}(z^d))
$$
and $\Sigma_{q,m}$ is defined analogously with $p_{r,c,u_m}$ replaced by
$q_{r,c,u_m}$. By Lemma~\ref{lem9}, we have that $\Sigma_{p,m+1}$ and
$\Sigma_{q,m+1}$ are also determined by $\Sigma_{p,m}$ and $\Sigma_{q,m}$
respectively.

It remains to show that all terms of $\Sigma_{p,m}$ and $\Sigma_{q,m}$ are
bounded by a constant independent of $m$. That in turn will imply that there
are only finitely many different values for $(\Sigma_{p,m}, \Sigma_{q,m})$
and there exist $m_1<m_2$ such that $\Sigma_{p,m_1} = \Sigma_{p,m_2},
\Sigma_{q,m_1} = \Sigma_{q,m_2}$, hence the sequence
$(r_{r,c,c,m})_{m\in\NN}$ is eventually periodic end the proof of
Proposition~\ref{prop2} is completed for this case.

Write the part $p_{r,c,u_m}/q_{r,c,u_m}$ of the convergent $p_{u_m}/q_{u_m}$
in the following form:
$$
\frac{p_{r,c,u_m}(z)}{q_{r,c,u_m}(z)} = \frac{\prod_{t=0}^{m-1} A^*_{t,m}(z) p^*_m(z)}{\prod_{t=0}^{m-1}B^*_{t,m}(z) q^*_m(z)},
$$
where the numerator and the denominator of the right hand side are coprime
and $A^*_{t,m}(z) \mid A_{r,c}(z^{d^t})$, $B^*_{t,m}(z)\mid
B_{r,c}(z^{d^t})$, $p^*_m(z) \mid p_{r,c,u}(z^{d^m})$ and $q^*_m(z) \mid
q_{r,c,u}(z^{d^m})$.

Since none of $A_{r,c}(z)$ and $B_{r,c}(z)$ are divisible by $\Phi_{r}(z)$ we
have that there exists $m_0\in\NN$ such that for all $m>m_0$, the polynomials
$A_{r,c}(z^{d^m})$ and $B_{r,c}(z^{d^m})$ are coprime with both $A_{r,c}(z)$
and $B_{r,c}(z)$. Therefore for each term $\sigma_{\zeta,m}$ of
$\Sigma_{p,m}$ we have
$$
\sigma_{\zeta,m} = \sum_{t=0}^{m_0} \sigma(\Phi_{r\zeta}(z), A^*_{t,m} (z^d)) + \sigma(\Phi_{r\zeta}(z), p^*_m(z^d)).
$$
By Lemma~\ref{lem11}, the right hand side is always bounded by some constant
independent of $m$. By analogous arguments, the same is true for all terms
$\sigma_{\zeta,m}$ of $\Sigma_{q,m}$.

{\bf Case 3.} Assume that $\Phi_r(z)$ divides $A_{r,c}(z)$. Then, since
$A_{r,c}(z)$ and $B_{r,c}(z)$ are coprime, we have that $\Phi_r(z)$ does not
divide $B_{r,c}(z)$.

Note that the case $\Phi_r(z)\mid B_{r,c}(z)$ can be handled analogously. We
just swap $A_{r,c}$ with $B_{r,c}$ and $p_{r,c,u}$ with $q_{r,c,u}$.
Therefore Case~3 is the last one which needs to be investigated.

\begin{lemma}\label{lem12}
For any $n\in\NN$ with $r(n,d) = r$ there exists $m\in\NN$ such that
$\Phi_n(z)\mid \Phi_r(z^{d^m})$.
\end{lemma}

\begin{proof}
We write $n$ as a product $n = rs$ and prove the lemma by induction on $s$.
For $s=1$ the statement is straightforward. Consider $S$ such that
$\rad(S)\mid \rad(d)$. Assume that the statement is true for all $s<S$ with
$\rad(s)\mid \rad(d)$ and prove it for $S$. Write the prime factorisations of
$S$ and $d$ in the following way:
$$
S = p_1^{\beta_1}\cdots p_k^{\beta_k}p_{k+1}^{\beta_{k+1}}\cdots p_{k+l}^{\beta_{k+l}};\quad d = p_1^{\alpha_1}\cdots p_{k+l}^{\alpha_{k+l}},
$$
where $\beta_1<\alpha_1,\ldots, \beta_k<\alpha_k$ and $\beta_{k+1}\ge
\alpha_{k+1}>0,\ldots, \beta_{k+l}\ge\alpha_{k+l}>0$. Then, by
Lemma~\ref{lem9}, one has that $\Phi_S(z)$ divides $\Phi_s(z^d)$ for $s =
p_{k+1}^{\beta_{k+1}-\alpha_{k+1}}\cdots p_{k+l}^{\beta_{k+l}-\alpha_{k+l}}$.
By induction assumption, we have that there exists $m$ such that
$\Phi_s(z)\mid \Phi_r(z^{d^m})$. Therefore, $\Phi_S(z)\mid
\Phi_r(z^{d^{m+1}})$.
\end{proof}

Similarly to Case~2, define the set $\SSS$ and the following tuple:
$$
\Sigma_{q,m}:=(\sigma_{\zeta_1,m}, \ldots, \sigma_{\zeta_N,m}),\quad\mbox{where}\; N=\#\SSS,\;\; \zeta_1,\ldots, \zeta_n\in\SSS\;\mbox{and}
$$
$$
\sigma_{\zeta,m}:=\sigma\left(\Phi_{r\zeta}(z),\prod_{t=0}^{m-1}B_{r,c}(z^{d^t})\cdot
q_{r,c,u}(z^{d^m})\right).
$$
As in Case~2, we have that all terms in $\Sigma_{q,m}$ are bounded by a
constant, which is independent of $m$. On the other hand, by
Lemma~\ref{lem9}, every polynomial $A_{r,c}(z^{d^t})$ is divisible by
$\Phi_r(z)$ and therefore
$$
\sigma\left(\Phi_{r}(z),\prod_{t=0}^{m-1}A_{r,c}(z^{d^t})\right)\ge m.
$$
In view of Lemma~\ref{lem12}, there exists $m_0$ big enough, so that for any
$\zeta\in\SSS$ and $m\ge m_0$ the value
$$
\sigma\left(\Phi_{r\sigma}(z),\prod_{t=0}^{m-1}A_{r,c}(z^{d^t})\cdot
p_{r,c,u}(z^{d^m})\right)
$$
is bigger than every term in $\Sigma_{q,m}$. That implies that for every
$m>m_0$ every polynomial $(\Phi_{r\zeta}(z))^{\sigma_{\zeta,m}}$ cancels out
in the expression
$$
\frac{p_{r,c,u_m}(z)}{q_{r,c,u_m}(z)} = \frac{\prod_{t=0}^{m-1} A_{r,c}(z^{d^t}) p_{r,c,u}(z^{d^m})}{\prod_{t=0}^{m-1}B_{r,c}(z^{d^t}) q_{r,c,u}(z^{d^m})}.
$$
Hence for $m\ge m_0$ the polynomial $q_{r,c,u_m}(z^d)$ is coprime with
$A_{r,c}(z)$, $B_{r,c}(z)$ divides $p_{r,c,u_m}(z^d)$ and therefore the value
$r_{r,c,c,m}$ is equal to $r_b$. Again we have that the sequence
$(r_{r,c,c,m})_{m\in\NN}$ is eventually periodic.

To finish the proof of Proposition~\ref{prop2} we observe that the sequence
$r_{c,m}$ is the sum of finitely many eventually periodic sequences and hence
is eventually periodic itself. \endproof

%Finally, we split the convergents $p_{u_n}(z)/q_{u_n}(z)$ into cyclotomic and
%non-cyclotomic parts:
%$$
%\frac{p_{c,m}(z)}{q_{c,m}(z)}:= \prod_{t=0}^{m-1}\frac{A_c(z^{d^t})}{B_c(z^{d^t})} \cdot \frac{p_{c,u}(z^{d^m})}{q_{c,u}(z^{d^m})};\quad
%\frac{p_{n,m}(z)}{q_{n,m}(z)}:= \prod_{t=0}^{m-1}\frac{A_n(z^{d^t})}{B_n(z^{d^t})} \cdot \frac{p_{n,u}(z^{d^m})}{q_{n,u}(z^{d^m})}.
%$$
%Clearly, the roots of $p_{c,m}$ and $q_{c,m}$ are all roots of unity, while
%the roots of $p_{n,m}$ and $q_{n,m}$ are not. Therefore $p_{c,m}$ and
%$q_{n,m}$ as well as $p_{n,m}$ and $q_{c,m}$ are always coprime and hence we
%can split the values $r_{c,n}$ into the sum of two parts: one comes from the
%cancellations of the cyclotomic part

\section{Application: $d=3$, infinite products of quadratic polynomials}

Consider the set of Mahler functions $g_\va(z) = g_{a_1,a_2} (z)$ which
satisfy the equation
$$
g_\va (z) = (z^2 +a_1z + a_2) q_\va(z^3);\quad a_1,a_2\in\ZZ.
$$
Such functions and their corresponding Mahler numbers were considered
in~\cite{badziahin_2017} and it was conjectured that, given $b\in\ZZ$ with
$|b|\ge 2$, if $g_\vu(b)\not \in\QQ$ then $\mu(g_\va(b))=2$ for all
$\va\in\ZZ^2$, except the following three families:
\begin{itemize}
\item[(a)] $\va = (s,s^2)$, $s\in\ZZ$;
\item[(b)] $\va = (s^3, -s^2(s^2+1))$, $s\in\ZZ$;
\item[(c)] $\va = (\pm2,1)$.
\end{itemize}

In~\cite[Theorem 9]{badziahin_2017} the lower bounds for the irrationality
exponents of $g_\va(b)$ for those families is provided. Here we demonstrate
how Theorem~\ref{th1} together with Lemmata~\ref{lem6} and~\ref{lem7} can be
used to show that lower bounds in~\cite{badziahin_2017} are sharp.

\medskip
\noindent{\bf Family (a).} Let $\va = (s,s^2)$. Simple calculations reveal
that the first convergent of $g_\va(z)$ is $1/(z-s)$ and
$$
(z-s)g_\va(z) - 1 = (s-s^3)z^{-3}+\ldots
$$
Therefore for $s^3-s\neq 0$ we have that $\Phi(g_\va)$ contains a primitive
gap $[1,3]$ of size 2. Note that $z^{3^m}-s$ is always coprime with the
polynomial $z^2 + sz+s^2$. Indeed, each root $z_0$ of the latter quadratic
polynomial satisfies $|z_0|^3 = |s|^3$, so $|z_0|^{3^m} = |s|^{3^m} > |s|$ as
soon as $|s|\ge 2$. But the last condition is equivalent to $s\neq \pm 1,0$
which in turn is equivalent to $s^3 - s\neq0$.

We thus have that the numerator and the denominator of
$$
\frac{\prod_{t=0}^{m-1} (z^{2\cdot 3^t} + sz^{3^t}+s^2)}{z^{3^m} - s}
$$
are always coprime. Therefore the all values $r_{c,m}$ equal zero and
equations~\eqref{form_un} imply that for the gaps $[u_n,v_n]$ generated by
$[1,3]$ we have
$$
\frac{v_{n+1}}{u_{n+1}} = \frac{3v_n -2}{3u_n}
$$
and therefore
$$
\liminf_{n\to \infty}\frac{v_n}{u_n} = \frac{v_0 - \frac{2}{2}}{u_0 + \frac{0}{2}} = 2.
$$

From Lemma~\ref{lem6} we know that the size of any primitive gap in
$\Phi(g_\va)$ does not exceed 5. Therefore, by Lemma~\ref{lem7}, only gaps
with
$$
2<\frac{v}{u-1}\le \frac{u+5}{u-1}
$$
may contribute to the irrationality exponent of $g_\va(b)$. The last
inequality is equivalent to $u<7$. It remains to note that $\Phi(g_\va) =
\{1,3,7,\ldots\}$ where the gap $[3,7]$ is not primitive and is generated by
$[1,3]$. Hence there are no other primitive gaps $[u,v]$ with $u<7$ and
Theorem~\ref{th1} implies:

\smallskip\noindent{\it Let $\va = (s,s^2)\in\ZZ^2$ with $s^3+s\neq 0$. If $|b|\ge 2$, $b\in\ZZ$ and $g_\va(b)\not\in\QQ$ then
$\mu(g_\va(b)) = 3$.}

For the remaining values of $s$ we have:
$$
g_{0,0}(z) = \frac{1}{z};\quad g_{1,1}(z) = \frac{1}{z-1};\quad g_{-1,1}(z) = \frac{1}{z+1}.
$$
The function $g_\va$ is then rational, and therefore $g_\va(b)\in\QQ$.

\medskip{\bf Family (b).} Let $\va = (s^3,-s^2(s^2+1))$. In this case we
compute
$$
\frac{p_2(z)}{q_2(z)} = \frac{z+s(s^2+1)}{z^2+sz+s^2}
$$
and
$$
q_2(z) g_\va(z) - p_2(z) = -(s^6+s^4+s^2)z^{-5} + \ldots
$$
Therefore for $s^6+s^4+s^2\neq 0$ we have that $\Phi(g_\va)$ contains the
primitive gap $[2,5]$ of size 3. One can easily check that $z^2 + s^3z -
s^2(s^2+1)= (z-s)(z+(s^3+s))$. On the other hand, all roots of $z^{2\cdot
3^m} + sz^{3^m} + s^2$ for $s\neq 0$ are not real. Therefore the fraction
$$
\frac{\prod_{t=0}^{m-1} (z^{2\cdot 3^t} + s^3z^{3^t}-s^2(s^2+1))p_2(z^{3^m})}{q_2(z^{3^m})}
$$
is always in its reduced form, i.e. every term of $r_{c,m}$ is zero. This
yields to
$$
\liminf_{n\to \infty}\frac{v_n}{u_n} = \frac{v_0 - \frac{2}{2}}{u_0 + \frac{0}{2}} = 2.
$$

As for the collection~(a), we need to check that $\Phi(g_\va)$ does not
contain any other primitive gap $[u,v]$ with $u<7$ which is obvious
(by~\eqref{form_un}, we have the big gaps $[2,5]$ and $[6,13]$ in
$\Phi(g_\va)$. There is no more space for big gaps with $u<7$). Therefore we
finally get:

\smallskip\noindent{\it Let $\va = (s^3,-s^2(s^2+1))$ with $s\in \ZZ, s^6+s^4+s^2\neq 0$. If $|b|\ge 2$, $b\in\ZZ$ and $g_\va(b)\not\in \QQ$ then
$\mu(g_\va(b)) = 3$.}

Finally notice that the equation $s^6+s^4+s^2=0$ has only one integer
solution: $s=0$. But $g_{0,0}(z)$ has already been considered in Family~(a)
and is equal to $1/z$.

\medskip\noindent{\bf Family (c).} Let $\va = (2,1)$. The case $\va =
(-2,1)$ is considered analogously and is left to the reader. One can check
that $1,2,3,4$ and $5$ belong to $\Phi(g_\va)$. Direct computation shows that
$[5,8]$ is the primitive gap in $\Phi(g_\va)$ and the corresponding fifth
convergent of $g_\va$ is
$$
p_5(z) = z^4+z^3+2z^2+4\;\mbox{and}\; q_5(z) = z^5-z^4+z^3-z^2+z-1 =
(z-1)(z^2+z+1)(z^2-z+1).
$$
In other words, $q_5(z) = \Phi_1(z)\Phi_3(z)\Phi_6(z)$. Lemma~\ref{lem9}
implies that all cyclotomic divisors of $q_5(z^{3^m})$ are either of the form
$\Phi_{3r}(z)$ with some integer $r$ or $\Phi_1(z)$. Hence $q_5(z^{3^m})$ is
always coprime with $z^2+2z+1 = \Phi_2(z)^2$, i.e. the fraction
$$
\frac{\prod_{t=0}^{m-1} (z^{2\cdot 3^t} + 2z^{3^t}+1)p_5(z^{3^m})}{q_5(z^{3^m})}
$$
is always in its reduced form and every term of $r_{c,m}$ is zero. This
yields to
$$
\liminf_{n\to\infty} \frac{v_n}{u_n} = \frac{v_0-1}{u_0} = \frac75.
$$

Now from Lemma~\ref{lem6} we know that the size of any primitive gap in
$\Phi(g_\va)$ does not exceed 5. Therefore, by Lemma~\ref{lem7}, only gaps
with
$$
\frac75<\frac{v}{u-1}\le \frac{u+5}{u-1}
$$
may contribute to the irrationality exponent of $g_\va(b)$. The last
inequality is equivalent to $u<16$. It remains to show that all integers from
8 to 15 belong to $\Phi(g_\va)$. This for example can be done
(see,~\cite[Corollary 1]{badziahin_2017} for justification) by checking that
the corresponding Hankel determinants
$$
H_n:= \det (c_{i+j-1})_{i,j\in\{1,\ldots, n\}},\quad n=8,\ldots 15
$$
are not zero, where $c_i$ are the coefficients of the series $g_\va$:
$$
g_\va(z) = \sum_{i=1}^\infty c_i z^{-i}.
$$

Finally we have:
\smallskip\noindent{\it Let $\va = (\pm 2,1)$. If $|b|\ge 2$, $b\in\ZZ$ then
$\mu(g_\va(b)) = \frac{12}{5}$.}

\end{document}